\documentclass[12pt]{amsart}
\usepackage{amssymb}
\usepackage{amsmath}
\usepackage[hidelinks]{hyperref}
\usepackage{etoolbox}
\usepackage{enumitem}
\usepackage{xcolor}
\usepackage{tikz-cd}
\usepackage{manfnt}
\usepackage[normalem]{ulem}
\hypersetup{
	colorlinks,
	linkcolor={red!30!black},
	citecolor={blue!50!black},
	urlcolor={blue!80!black}
}
\makeindex

\newtheorem{theorem}{Theorem}[section]
\newtheorem{lemma}[theorem]{Lemma}

\theoremstyle{remark}
\newtheorem{remark}[theorem]{Remark}

\newcommand{\F}{\mathbb{F}}
\newcommand{\Z}{\mathbb{Z}}
\newcommand{\Q}{\mathbb{Q}}

\newcommand{\C}{\mathbb{C}}
\newcommand{\Cc}{\mathbb{C}^{\times}}

\newcommand{\cC}{\mathcal{C}}

\newcommand{\cO}{\mathcal{O}}

\newcommand{\bbG}{\mathbb{G}}
\newcommand{\bbH}{\mathbb{H}}

\newcommand{\fg}{\mathfrak{g}}
\newcommand{\fh}{\mathfrak{h}}

\newcommand{\fu}{\mathfrak{u}}

\newcommand{\ft}{\mathfrak{t}}

\newcommand{\fm}{\mathfrak{m}}

\newcommand{\bsl}{\backslash}
\newcommand{\ra}{\rightarrow}
\newcommand{\ira}{\hookrightarrow}

\newcommand{\xra}{\xrightarrow}

\newcommand{\ol}{\overline}
\newcommand{\what}{\widehat}

\newcommand{\Lie}{\operatorname{Lie}}
\newcommand{\rank}{\operatorname{rank}}

\newcommand{\cha}{\operatorname{char}}
\newcommand{\Ad}{\operatorname{Ad}}

\newcommand{\id}{\mathrm{id}}

\newcommand{\cci}{\mathcal{C}_c^{\infty}}

\newcommand{\Hom}{\operatorname{Hom}}

\newcommand{\supp}{\operatorname{supp}}

\newcommand{\FT}{\operatorname{FT}}

\newcommand{\SO}{\mathrm{SO}}

\newcommand{\Sp}{\mathrm{Sp}}
\newcommand{\GL}{\mathrm{GL}}

\newcommand{\Grs}{G^{\mathrm{rs}}}
\newcommand{\Gss}{G^{\mathrm{ss}}}

\newcommand{\fhrs}{\mathfrak{h}^{\mathrm{rs}}}
\newcommand{\fhss}{\mathfrak{h}^{\mathrm{ss}}}

\newcommand{\ftr}{\mathfrak{t}^{\mathrm{r}}}
\newcommand{\der}{\mathrm{der}}

\begin{document}

\title{On local integrability results for $p$-adic reductive groups}
\author{Cheng-Chiang Tsai}
\thanks{The author is supported by NSTC grant 114-2115-M-001-009, 114-2628-M-001-003, and AS grant AS-IV-115-M07.}
\email{chchtsai@as.edu.tw}
\address{Institute of Mathematics, Academia Sinica, 6F, Astronomy-Mathematics Building, No. 1,
Sec. 4, Roosevelt Road, Taipei, Taiwan \vskip.2cm
also Department of Applied Mathematics, National Sun Yat-Sen University, and Department of Mathematics, National Taiwan University}
\begin{abstract}
    We present a short proof, based on local character expansions, of the celebrated theorem of Harish-Chandra about local integrability of complex characters of $p$-adic reductive groups. The proof gives an algebraic incarnation of the local integrability that works for some coefficients different from $\mathbb{C}$, verifies local integrability in cases that appear not covered in the literature, and shows that a character is locally-$L^{\alpha}$ for some specified $\alpha>1$ as in \cite{GGH23}.
\end{abstract}
\makeatletter
\patchcmd{\@maketitle}
{\ifx\@empty\@dedicatory}
{\ifx\@empty\@date \else {\vskip3ex \centering\footnotesize\@date\par\vskip1ex}\fi
    \ifx\@empty\@dedicatory}
{}{}
\patchcmd{\@adminfootnotes}
{\ifx\@empty\@date\else \@footnotetext{\@setdate}\fi}
{}{}{}
\makeatother
\maketitle 

\section{Introduction}\label{sec:main}
Let $F$ be a non-archimedean local field. We fix a reductive group $\bbG$ over $F$ (possibly disconnected). Fix also a non-trivial additive character $\psi:(F,+)\ra\Cc$; for ease of reading we will write our proof in the traditional language of complex coefficients, but explain in Remark \ref{rmk:alg} that we can work with fields such as $\ol{\Q_{\ell}}$, and, for exponentially large $\ell$, with $\ol{\F_{\ell}}$.
Let $\pi$ be a finitely generated admissible $\C$-representation of $G:=\bbG(F)$. For $G$ and other $p$-adic reductive groups we will fix a Haar measure on each of them, as well as an translation-invariant measure on each vector space over $F$; the choices will never matter.

We assume Hypotheses \ref{D}, \ref{C}, \ref{A} and \ref{B} in \S\ref{sec:hyp}, which states roughly that local character expansions hold, and either $\cha(F)=0$ or $p$ does not divide $|W_G|$. Our goal is to prove the celebrated theorem \cite[Theorem 16.3]{HC99} of Harish-Chandra about its character $\Theta_{\pi}$ \cite[Preface, p. x]{HC99} that

\begin{enumerate}[label=(\Roman*)]
    \item\label{1} The distribution $\Theta_{\pi}|_{\Grs}$ is represented \cite[p. 1, (\dag)]{HC99} by a locally constant function $F_{\pi}\in C^{\infty}(\Grs)^G$.
    \item\label{3} The function $F_{\pi}$ is locally-$L^1$ on $G$.
    \item\label{4} The distribution $\Theta_{\pi}$ is represented by $F_{\pi}$.
\end{enumerate}

Consider any $s\in\Gss$ and $\bbH:=Z_{\bbG}(s)$. We write $H=\bbH(F)=Z_G(s)$ and $\fh=\Lie H=(\Lie\bbH)(F)$, and $\fh^*:=\Hom_F(\fh,F)$. For any nilpotent orbit $\cO^*\subset\fh^*$, we consider $\hat{I}^H_{\cO^*}$ the Fourier transform of the nilpotent orbital integral. We prove

\begin{enumerate}[label=(\roman*)]
    \item\label{i} The distribution $\hat{I}^H_{\cO^*}|_{\fhrs}$ is represented by a locally constant function $\hat{J}^H_{\cO^*}\in C^{\infty}(\fhrs)^H$.
    \item\label{iii} $\hat{J}^H_{\cO^*}$ is locally-$L^1$ on $\fh$.
    \item\label{iv} The distribution $\hat{I}^H_{\cO^*}$ is represented by the function $\hat{J}^H_{\cO^*}$.
\end{enumerate}

\section{The proof}

\begin{enumerate}[label=Step \arabic*.]
    \item\label{b} Apply local character expansion, namely Hypothesis \ref{A} (resp. \ref{B}), to $\Theta_{\pi}$ (resp. $\hat{I}^H_{\cO^*}$) and $s\in\Grs$ (resp. $X_s=X\in\fhrs$). The nilpotent cone of $Z_G(s)$ (resp. $Z_H(X_s)$) is exactly $\{0\}$ and the Fourier transform of any distribution supported on $\{0\}$ is constant. Therefore $\Theta_{\pi}$ (resp. $\hat{I}^H_{\cO^*}|_{\fhrs}$) is locally constant near $s$ (resp. $X$). This proves \ref{1} (resp. \ref{i}).
    \item\label{d} The assertions \ref{3}, \ref{4}, \ref{iii} and \ref{iv} are all local statements. For any element $g\in G$ (resp. $X\in\fh$) and any neighborhood of its semisimple part $g_s\in G$ (resp. $X_{s}\in\fh$), we have that $g$ (resp. $X$) is conjugate to an element in the neighborhood. Thus
    it suffices to prove \ref{3} and \ref{4} locally near every semisimple $s\in \Gss$, as well as \ref{iii} and \ref{iv} locally near every semisimple $X\in\fhss$.
    \item\label{a}  Thanks to Hypothesis \ref{A}, near any semisimple $s\in G$ the distribution $\Theta_{\pi}$ can be identified with a linear combination of nilpotent orbital integrals on $Z_{\fg}(s)^*$. Thus \ref{3} and \ref{4} follow from \ref{iii} and \ref{iv} for $\bbH=Z_{\bbG}(s)$.
    \item Thanks to Hypothesis \ref{D}\ref{hyp:ss}, we may assume that $H$ is semisimple.
    \item\label{f} By induction (and Hypothesis \ref{D}\ref{hyp:Z-Z}) we assume \ref{iii} and \ref{iv} for groups of the form $Z_{\bbH}(X)\subsetneq\bbH$ for $X\in\fhss$. Hypothesis \ref{B} then implies that \ref{iii} and \ref{iv} hold near any non-central (i.e. non-zero) $X\in\fhss$. 
    Equivalently, we know \ref{iii} and \ref{iv} near any element with a non-zero semisimple part. It remains to prove the statement near $0\in\fh$.
\end{enumerate}
Let $T\subset H$ be a maximal $F$-torus, $\ft:=\Lie T$ and $\ftr:=\fhrs\cap\ft$. For $Y\in\ftr$ denote by $I_Y^H(f)=\int_{H/T}f(\Ad(h)Y)\frac{dh}{dt}$. Consider the following statement:
\begin{enumerate}[label=(\roman*)$_T$]
\setcounter{enumi}{1}
    \item\label{iiiT} For any $f\in C_c^{\infty}(\fh)$, the function on $\ftr$ given by
    \[
    Y\mapsto|D_{\fh}(Y)|\cdot\hat{J}^H_{\cO^*}(Y)\cdot I_Y^H(f)
    \] 
    is locally-$L^1$ on $\ft$.
\end{enumerate}
\begin{enumerate}[label=Step \arabic*.]
\setcounter{enumi}{5}
    \item For any $f\in C_c^{\infty}(\fh)$, the Weyl integration formula implies that
    \begin{equation}\label{eq:Weyl}
    \int_{\fhrs}\hat{J}_{\cO^*}^H(X)\cdot f(X)dX=\sum_T\frac{1}{|N_H(T)/T|}\int_{\ftr}|D_{\fh}(Y)|\cdot\hat{J}^H_{\cO^*}(Y)\cdot I_Y^H(f)dY
    \end{equation}
    over a set of representatives $\{T\}$ of $H$-conjugacy classes of maximal $F$-tori in $H$, and
    \[\int_{\fhrs}|\hat{J}_{\cO^*}^H(X)|\cdot |f(X)|dX=\sum_T\frac{1}{|N_H(T)/T|}\int_{\ftr}|D_{\fh}(Y)|\cdot|\hat{J}^H_{\cO^*}(Y)|\cdot I_Y^H(|f|)dY
    \]
    Therefore \ref{iii} is equivalent to \ref{iiiT} holding for every maximal $F$-torus $T\subset H$. More precisely, \ref{iii} holding near $X\in\fhss$ is equivalent to \ref{iiiT} holding near $X$ for every maximal $F$-torus $T\subset H$ with $X\in\ft$. In particular, \ref{f} implies that \ref{iiiT} hold on $\ft^r\bsl\{0\}$.
    It remains to prove \ref{iiiT} near $0\in\ft$.
    \item\label{k} Choose a uniformizer $\varpi\in F^{\times}$ with $|\varpi|=q^{-1}$. To prove \ref{iiiT} near $0$, consider a sufficiently small lattice $\Lambda_{\ft}\subset\ft$. 
    For $Y\in\Lambda_{\ft}\cap\ft^r$ we have Shalika germ expansion \cite[Theorem 2.1.1]{Sha72}
    \[
    I_Y^H(f)=\sum_{\cO}\Gamma_{\cO}(Y)\cdot I_{\cO}^H(f).
    \]
    and a scaling property \cite[\S17]{Kot05} (and \cite[Theorem A]{Pre03} when $\cha(F)>0$ under Hypothesis \ref{D}) that
    \[\begin{array}{l}|D_{\fh}(\varpi^2 Y)|=q^{-2\dim H+2\dim T}\cdot |D_{\fh}(Y)|,\\
    \hat{J}^H_{\cO^*}(\varpi^2 Y)=q^{\dim\cO^*}\cdot\hat{J}^H_{\cO^*}(Y),\\
    \Gamma_{\cO}(\varpi^2Y)=q^{\dim\cO}\cdot\Gamma_{\cO}(Y).\end{array}\]
    Consequently, for every $Y\in\Lambda_{\ft}\cap\ftr$ we have
    \[
    \begin{array}{lll}
    &&|D_{\fh}(\varpi^2 Y)|\cdot\hat{J}_{\cO^*}^H(\varpi^2 Y)\cdot \Gamma_{\cO}(\varpi^2 Y)\\
    =q^{-2\dim H+2\dim T+\dim\cO^*+\dim\cO}&\times&|D_{\fh}(Y)|\cdot\hat{J}_{\cO^*}^H(Y)\cdot\Gamma_{\cO}(Y)
    \end{array}
    \]
    In other words, the integrand on the RHS of \eqref{eq:Weyl} can be separated according to $\dim\cO$, so that each term has its own scaling. Since $\dim\cO^*+\dim\cO<2\dim H$, we are now in the situation of the following elementary lemma which implies \ref{iiiT} and thus \ref{iii}.
    \begin{lemma} Let $\Lambda$ be a free $\cO_F$-module of rank $r<\infty$, and $U\subset\Lambda$ an open subset satisfying $Y\in U\implies\varpi^2Y\in U$. Suppose we have integers $d_1<d_2<...<d_n<2r$ and functions $f_1,...,f_n\in C_c^{\infty}(U)$ satisfying
    \[
    f_j(\varpi^2Y)=q^{d_j}\cdot f_j(Y),\;\forall Y\in\Lambda,\;j=1,...,n.
    \]
    Suppose moreover that $\sum_{j=1}^n f_j$ is $L^1$ on  $\varpi^{2i}\Lambda\bsl\varpi^{2i+2}\Lambda$ for every $i\in\Z_{\ge 0}$. Then all $f_j$ are $L^1$ on $\Lambda$.
    \end{lemma}
    \begin{proof} Firstly any $f_j$ can be interpolated from $\sum f_j$, so that each $f_j$ is $L^1$ on $\varpi^{2i}\Lambda\bsl\varpi^{2i+2}\Lambda$. Then each $f_j$ is $L^1$ on $\Lambda$ because the geometric progression $\left\|f_j\right\|_{\varpi^{2i}\Lambda\bsl\varpi^{2i+2}\Lambda}$ has ratio $<1$.
    \end{proof}
    \item Having proved \ref{iii}, we have that $\hat{J}^H_{\cO^*}\in C^{\infty}(\fhrs)^H$ defines a distribution on $\fh$, which we again denote by $\hat{J}^H_{\cO^*}$. We have to prove that $\hat{J}^H_{\cO^*}=\hat{I}^H_{\cO^*}$. By our induction on $\bbH$, the two distributions agree near any non-zero semisimple element. Hence their difference $\hat{I}^H_{\cO^*}-\hat{J}^H_{\cO^*}$ is supported on the nilpotent cone, and therefore $\hat{I}^H_{\cO^*}-\hat{J}^H_{\cO^*}=\sum c_{\cO}I_{\cO}^H$ for some $c_{\cO}\in\C$.
    \item\label{m} Write $f_{\varpi^2}(Y):=f(\varpi^2Y)$ for $f\in C_c^{\infty}(\fh)$.
    Any $I_{\cO}^H$ satisfies 
    \[I_{\cO}(f_{\varpi^2})=q^{\dim\cO}I_{\cO}(f).\] Meanwhile, we have $\widehat{f_{\varpi^2}}=q^{2\dim H}\cdot(\hat{f})_{\varpi^{-2}}$. Therefore
    \[
    \begin{array}{c}
    \hat{I}^H_{\cO^*}(f_{\varpi^2})=q^{2\dim H-\dim\cO^*}\cdot\hat{I}_{\cO^*}^H(f).\\
    \hat{J}^H_{\cO^*}(f_{\varpi^2})=q^{2\dim H-\dim\cO^*}\cdot\hat{J}_{\cO^*}^H(f).
    \end{array}
    \]
    Hence any $\cO$ with $c_{\cO}\not=0$ satisfies $\dim\cO=2\dim H-\dim\cO^*$, which is absurd because (again!) $\dim\cO+\dim\cO^*<2\dim H$. Hence $\hat{I}^H_{\cO^*}-\hat{J}^H_{\cO^*}=0$, and we have proved \ref{iv}.
\end{enumerate}

\section{Remarks}

\begin{remark} A similar argument shows that $\hat{J}^H_{\cO^*}$ is locally bounded by $O(|D_{\fh}|^{-1/2})$. To see this, fix a maximal $F$-torus $T\subset H$. By induction $\hat{J}^H_{\cO^*}$ is locally bounded by $O(|D_{\fh}|^{-1/2})$ near any non-central semisimple element in $\ft$. By scaling as in \ref{k} the bound then holds on $\Lambda_{\ft}\cap\ftr$.
\end{remark}

\begin{remark}\label{rmk:alg} Suppose $\cC$ is any field for which
\begin{enumerate}[label=\arabic*.]
    \item We have a non-trivial homomorphism $\psi:(F,+)\ra \cC^{\times}$ with open kernel. In particular $\cha(\cC)\not=p$.
    \item\label{cond:2} We have $q^{2n}\not=1$ in $\cC$ for every $0<n\le\dim\bbG^{\mathrm{der}}$.
    \item\label{cond:3} $\cha(\cC)$ does not divide $|W_{\bbG}|$. (It is easy to see that \ref{cond:2}$\implies$\ref{cond:3}.)
\end{enumerate} 
We write $\cC^{\infty}(\Grs)$, etc., to emphasize $\cC$-valued locally constant functions on $\Grs$. Consider a finitely generated admissible $\cC$-representation $\pi$, its character $\Theta_{\pi}:\cci(G)\ra\cC$, and Fourier transforms of $\cC$-valued orbital integrals (in the sense of e.g. \cite[Corollary 3.39]{Tsa26}). Our method constructs the distribution $\Theta_{\pi}$ (resp. $\hat{I}_{\cO^*}^H$) in terms of $F_{\pi}\in \cC^{\infty}(\Grs)^G$ (resp. $\hat{J}_{\cO^*}^H\in \cC^{\infty}(\fhrs)^H$) using geometric series. 

This can be seen as follows: Fix any $f\in\cci(\fh)$. On the RHS of \eqref{eq:Weyl}, we would like to realize $|D_{\fh}(Y)|\cdot\hat{J}^H_{\cO^*}(Y)\cdot I_Y^H(f)dY$ as a measure on $\ft$, in the sense that it is a map
\[\mu^T_f:\{\text{open compact subsets $E\subset\ft$}\}\to\cC\]
that is additive in the sense $\mu^T_f(E_1\sqcup E_2)=\mu^T_f(E_1)+\mu^T_f(E_2)$. (When $\cC=\C$ this was done using traditional absolute convergence.) A measure can be defined locally and glued together, and by induction using Hypothesis \ref{B} we may assume that $\mu^T_f$ has been defined everywhere locally except near $\{0\}=Z(\fh)\subset\ft$. 
Then, by \ref{k}, $\mu^T_f$ can be defined locally near $0\in\ft$ using geometric series (of ratio $\not=1$ in $\cC$). The support of $\mu^T_f$ is evidently compact, and we define the LHS of \eqref{eq:Weyl} as $\hat{J}_{\cO^*}^H(f):=\sum_T\frac{1}{|N_H(T)/T|}\mu^T_f(E^T)\in\cC$ for any collection of $E^T\supset\supp(\mu^T_f)$. By \ref{m} we have $\hat{I}_{\cO^*}^H(f)=\hat{J}_{\cO^*}^H(f)$ for all $f\in\cci(\fh)$. The group case then follows from Hypothesis \ref{A}.



When $\cha(C)=\ell$ is positive and large enough, the fact that $\Theta_{\pi}$ and $\hat{I}_{\cO^*}$ are determined by their restrictions to regular semisimple elements was shown in \cite[D.4 Th\'{e}or\`{e}me]{VW01}. Their method likely gives much better bounds on $\ell$. As noted {\it loc. cit.}, the restrictions no longer determine the distributions when e.g. $\ell=2$ and $\bbG=\GL_2$.
\end{remark}

\begin{remark} For classical groups, the Cayley transform substituted for the exponential map. Hence we conclude that $\Theta_{\pi}$ is represented by a locally-$L^1$ function when $G=\mathrm{U}_{n}(\F_q((t)))$, $\SO_{2n+1}(\F_q((t)))$, $\Sp_{2n}(\F_q((t)))$, and $\SO_{2n}(\F_q((t)))$ with $p>6n+4$, and also for isogenous groups such as Spin groups, using the local character expansion result from \cite[Corollary 12.10]{AK07}. (The cited corollary uses the language of locally constant functions, but its proof is totally at the level of distributions.) Integrability results when $\cha(F)>0$ were obtained in several works, including \cite[Theorem 2.1]{CGH14} and \cite[Theorem C, Proposition D and Theorem F]{AGKSGH26}; see \cite[\S1.2]{AGKSGH26} for a beautiful discussion.
\end{remark}

\begin{remark} The proof in \ref{k} shows that $\hat{J}_{\cO^*}^H$ is locally-$L^{\alpha}$ for any
\begin{equation}\label{eq:alphabound}
\alpha<\min_{M,\cO_M^*,\cO_M}\frac{2\dim M^{\mathrm{der}}-\dim\cO_M}{\dim\cO_M^*}
\end{equation}
where the minimum is taken over (not necessarily tame) twisted Levi subgroups $M\subset H$ and nilpotent orbits $\cO_M\subset\fm$ and $\cO_M^*\subset\fm^*$, with the constraint that $\cO_M^*\subset\ol{\cO^*}$ where $\ol{\cO^*}$ denotes the $p$-adic closure, see Lemma \ref{lem:closure} below.

Such $L^{\alpha}$-results for $\alpha>1$ were already obtained in \cite[Theorem A, B and C]{GGH23}. To the best of our knowledge, non-overlapping cases are treated. The results agree e.g. for generic representations / regular nilpotent orbits, for which both methods show the bound to be sharp. When $H=\GL_n$, the term on the RHS of \eqref{eq:alphabound} agrees with the term in \cite[Remark 1.7]{GGH23} for the Levi $M=\GL_{n_k}\times\GL_1^{n-n_k}\subset\GL_n$ appearing {\it loc. cit.}. We remark that no constraint for $\cO_M$ is available, except that when $M$ is not quasi-split the dimension becomes obviously smaller than that of regular orbits. 
\end{remark}

\section{Hypotheses}\label{sec:hyp}

Let $F$ and $\bbG$ be as before. The hypotheses  are \ref{D}, \ref{C}, \ref{A} and \ref{B} as follows:

\begin{enumerate}[label=(\Alph*)]
    \item\label{D} Either $\cha(F)=0$, or $\cha(F)=p$ does not divide the order of the absolute Weyl group of $\bbG$. By \cite[Proposition 48]{Mc04} this\footnote{The cited proposition was written for linear algebraic groups with $p>\rank_{\bar{F}}\bbG^{\der}+1$. The proof uses that $p$ is very good for certain semisimple subgroups of $\bbG$, and works without change under our hypothesis.} implies that Jordan decompositions always exist in $H=Z_G(s)$ for any $s\in\Gss$. In particular we do not treat the subtle case of $\GL_p(\F_p((t)))$ as in \cite{AGKSGH26}. We also have
    \begin{enumerate}[label=(\arabic*)]
        \item Separability of adjoint actions so that orbital integrals can be defined \cite[Lemma 44, Theorem 45]{Mc04}.
        \item The number of $G$-conjugacy classes of maximal $F$-tori in $G$ is finite.
        \item Weyl integration formula holds.
        \item Regular semisimple elements are Zariski open dense \cite[Proposition 3.8]{MO-Kno19, Moh03} for $\bbG$ as well as $\Lie\bbH$.
        \item $\fh\cong\fh^*$ as $H$-representations.
        \item\label{hyp:ss} $\fh\cong Z(\fh)\times\fh^{\mathrm{der}}$.
        \item\label{hyp:Z-Z} For any $X\in\fh^{\mathrm{ss}}$, there exists $s_X\in\Gss$ such that $Z_H(X)=Z_G(s_X)$.
    \end{enumerate}
\end{enumerate}


\begin{enumerate}[label=(\Alph*)]
\setcounter{enumi}{1}    \item\label{C} For any subgroup $H:=Z_G(s)$ for some $s\in \Gss$, we fix an $H$-equivariant map $\exp_H:\fu_H\xra{\sim} U_H$ where $\fu_H\subset\fh$ (resp. $U_H\subset H$) is an $H$-invariant open subset containing $0$ (resp. $1$), and such that $(d\exp_H)_0=\id_{\fh}$.
\end{enumerate}

\begin{remark} The map $\exp_H$ is well-known to exist when $\cha(F)=0$. Under Hypothesis \ref{D}, $\exp_H$ exists thanks to \cite[Lemma C.4]{BKV16}. Hypotheses \ref{A} and \ref{B} below (which are stated with $\exp_H$) are not supposed to depend on the choice of $\exp_H$, but we will not study this issue.
\end{remark}

The next hypotheses are about semisimple descents of $\Theta_{\pi}$ and $\hat{I}_{\cO}^G$, which is essentially \cite[Theorem 16.2]{HC99}, but for our need we will spell it as follows:

\begin{enumerate}[label=(\Alph*)]
\setcounter{enumi}{2}
    \item\label{A} Let $s\in\Gss$ and $M:=Z_G(s)$. Take $\fm^{\perp}$ to be the $M$-invariant complement of $\fm$ in $\fg$.
    There exist sufficiently small lattices (depending on $\Theta_{\pi}$) $\Lambda_{\fm}\subset\fm$ and $\Lambda_{\fm^{\perp}}\subset\fm^{\perp}$ such that 
    \[\begin{array}{cccccl}c_s:&\Lambda_{\fm}&\times&\Lambda_{\fm^{\perp}}&\ra& G\\
    &(Y&,&Y')&\mapsto&\Ad(\exp_G(Y'))(s\cdot\exp_G(Y))
    \end{array}\]
    is an open embedding of $F$-analytic manifolds, and an $M$-invariant distribution $\delta_M\in D(\fm^*)^M$ supported on the nilpotent cone of $\fm^*$, such that for any $f_{\fm}\otimes f_{\fm^{\perp}}\in C_c^{\infty}(\Lambda_{\fm})\otimes_{\C}C_c^{\infty}(\Lambda_{\fm^{\perp}})=C_c^{\infty}(\Lambda_{\fm}\times\Lambda_{\fm^\perp})$ we have
    \[
    \Theta_{\pi}((c_s)_*(f_{\fm^\perp}\otimes f_{\fm}))=\delta_\fm(\FT_{\fm}(f_\fm))\cdot\int_{\Lambda_{\fm^{\perp}}}f_{\fm^\perp}(Y')dY'
    \]
    where $\mathrm{FT}_{\fm}:C_c^{\infty}(\fm)\ra C_c^{\infty}(\fm^*)$ is the Fourier transform. In other words, the distribution $c_s^*(\Theta_{\pi})$ is the tensor product of $\what{\delta_{\fm}}$ with a constant distribution on $\Lambda_{\fm^{\perp}}$.
    \item\label{B} 
    Let $\bbH=Z_{\bbG}(s)$ for some $s\in \Gss$. Fix $X_s\in\fh^{\mathrm{ss}}$ and $M:=Z_H(s)$.  Take $\fm^{\perp}$ to be the $M$-invariant complement of $\fm$ in $\fh$. Let $\delta$ be any invariant distribution on the nilpotent cone of $\fh^*$, and $\hat{\delta}$ its Fourier transform. Then there exist sufficiently small lattices $\Lambda_{\fm}\subset\fm$ and $\Lambda_{\fm^{\perp}}\subset\fm^{\perp}$ such that 
    \[\begin{array}{cccccl}c_{X_s}:&\Lambda_{\fm}&\times&\Lambda_{\fm^{\perp}}&\ra&\fh\\
    &(Y&,&Y')&\mapsto&\Ad(\exp_H(Y'))(X_s+Y)
    \end{array}\]
    is an open embedding of $F$-analytic manifolds, and an $M$-invariant distribution $\delta_\fm\in D(\fm^*)^M$ supported on the nilpotent cone of $\fm^*$, such that for any $f_{\fm}\otimes f_{\fm^{\perp}}\in C_c^{\infty}(\Lambda_{\fm})\otimes_{\C}C_c^{\infty}(\Lambda_{\fm^{\perp}})=C_c^{\infty}(\Lambda_{\fm}\times\Lambda_{\fm^\perp})$ we have
    \[
    \delta(\FT_{\fh}\left((c_{X_s})_*(f_{\fm^\perp}\otimes f_{\fm})\right))=\delta_\fm(\FT_{\fm}(f_\fm))\cdot\int_{\Lambda_{\fm^{\perp}}}f_{\fm^\perp}(Y')dY'.
    \]
    In other words, the distribution $c_{X_s}^*(\what{\delta})$ is the tensor product of $\what{\delta_{\fm}}$ with a constant distribution on $\Lambda_{\fm^{\perp}}$.
\end{enumerate}

\begin{remark}\label{rmk:char0} Howe originally stated and proved his local character expansion for $\GL_n$ near the identity \cite[Proposition 3]{How74} at the level of distributions. Harish-Chandra stated local character expansions \cite[Theorem 16.2]{HC99} in terms of local integrability theorems. It was observed in \cite[Th\'{e}or\`{e}me E.4.4]{VW01} that Harish-Chandra's proof of the local character expansions directly works for mod-$\ell$ coefficients, and in particular does not rely on local integrability. We plan to supply some additional exposition on this in a forthcoming version of \cite{Tsa26}.
\end{remark}

\begin{lemma}\label{lem:closure} Let $X_s\in\fhss$ and $M:=Z_H(X_s)$. We identify $\fm^*\ira\fh^*$ as the orthogonal complement to $\fm^{\perp}\subset\fh$. The local character expansion of $\hat{I}^H_{\cO^*}$ near $X_s$ is a linear combination of $\hat{I}^M_{\cO_M^*}$ with $\cO_M^*\subset\ol{\cO^*}$.
\end{lemma}

\begin{proof} It suffices to consider an $\cO_M^*$ that is maximal in the local character expansion of $\hat{I}^H_{\cO^*}$ near $X_s$. Choose $X_M^*\in\cO_M^*$, and a neighborhood $U_M$ of $X_M^*$ in $\fm^*$ such that any nilpotent $\Ad^*(H)$-orbit meeting $U_M$ contains $\cO_M^*$ in its closure. Consider the test function $f_i:=\mathrm{FT}_{\fm}^{-1}(1_{\varpi^{-4i} U_M})\otimes 1_{\varpi^{3i}\Lambda_{\fm^{\perp}}}$ on $\Lambda_{\fm}\times\Lambda_{\fm^{\perp}}$.  For $i\gg 0$ we have
\begin{equation}\label{eq:LCEatXs}
\hat{I}_{\cO^*}^H((c_{X_s})_*f_i)=c_i\cdot I_{\cO_M^*}^M(1_{\varpi^{-4i} U_M})\not=0\text{, for some }c_i\not=0
\end{equation}
Denote by $l_{X_s}$ the translation by $X_s$, and $\psi_{X_s}$ the function on $\fm^*$ given by $X^*\mapsto\psi(\langle X_s,X^*\rangle)$. 
For $i\gg 0$, since $[\varpi^{3i}\Lambda_{\fm^{\perp}},\varpi^{3i}\Lambda_{\fm^{\perp}}]\subset\varpi^{4i}(\Lambda_{\fm}+\Lambda_{\fm^{\perp}})$, we have
\[
(c_{X_s})_*(\mathrm{FT}_{\fm}^{-1}(1_{\varpi^{-4i} U_M})\otimes 1_{\varpi^{3i}\Lambda_{\fm^{\perp}}})=(l_{X_s}\circ\mathrm{FT}_{\fm}^{-1}(1_{\varpi^{-4i} U_M}))\otimes 1_{\varpi^{3i}[X_s,\Lambda_{\fm^{\perp}}]}
\]
\[
\begin{array}{rl}
\implies&\mathrm{FT}_{\fh}\left((c_{X_s})_*(\mathrm{FT}_{\fm}^{-1}(1_{\varpi^{-4i} U_M})\otimes 1_{\varpi^{3i}\Lambda_{\fm^{\perp}}})\right)
\\=&(\psi_{X_s}\cdot1_{\varpi^{-4i} U_M})\otimes\mathrm{FT}_{\fm^{\perp}}\left(1_{\varpi^{3i}[X_s,\Lambda_{\fm^{\perp}}]}\right).
\end{array}
\]
By the property of $U_M$, for $i\gg 0$, the support of the last function intersects a nilpotent orbit $\cO^*$ only if $\cO_M^*\subset\ol{\cO^*}$. Hence \eqref{eq:LCEatXs} implies the lemma.
\end{proof}

\subsection*{Acknowledgment} This article grew out of a lecture series in Spring 2026 on characters of $p$-adic groups. I am very indebted to the audience, including Cailan Li, Chih-Kai Liao, Hao-An Wu, Hao-Ting Li, Jun-Hao Huang, Masao Oi, Shih-Yu Chen, Song-Yun Chen, Tzu-Yang Tsai, Wei-Hsuan Hsin, Wille Liu and Yu-Chun Chang, as well as the National Center for Theoretical Sciences in Taiwan, and the Harish-Chandra seminar at University of Michigan.

I am genuinely grateful to Stephen DeBacker for countless illuminating discussions about characters and very helpful suggestions on an earlier version. I very much thank Julia Gordan for a wonderful talk and subsequent discussions on local integrability. I really appreciate Aron Heleodoro, Jessica Fintzen, Marie-France Vign\'{e}ras, Masao Oi and Loren Spice for very inspiring conversations. Lastly, ChatGPT, Gemini and MathOverflow provided numerous helpful (counter-)examples, references and writing suggestions.

\bibliographystyle{amsalpha}
\bibliography{biblio.bib}
\end{document}
